\documentclass[12pt,papersize=A4]{article}
\usepackage[theorem/counter/prefix=section,
            xref/biblatex/backend=biber]{mac-bundle}

\numberwithin{equation}{main@theorem}

\newcommand{\alignhack}[2]{\mathrel{\makebox[\widthof{$#2$}]{$#1$}}} % make #1 the width of #2 for use in amsmath's accursed align environment

\newcommand \op {\mathrm{op}}

\newconstant \Set   {Set}
\newconstant \Space {Spc}
\newconstant \Cat   {Cat}
\newconstant \Groupoid {Gpd}
\newconstant \point {pt}

\newfunction \Presheaf {PSh}
\let \PSh \Presheaf
\newfunction \Fun {Fun}
\newfunction \Map {Map}
\newfunction \Sh {Sh}
\newfunction \id {id}
\newfunction \Nat {Nat}
\newfunction \evaluation {ev}
\newfunction \constant {const}
\newfunction \forget {forget}
\newfunction \SemiRepresentable {SR}
\newfunction \Hom {Hom}
\newfunction \idx {idx}
\newfunction \projection {pr}
\newfunction \coCart {coCart}

\newcommand \nerve {\tilde{N}}

\newcommand \slice {\mathrel{/}}

\newcommand \opens {\mathcal{U}}

\newcommand \senv {\push_{\Delta_+/\Delta}}
\newcommand \LeftFibration {\mathtt{l-fib}}

\newfunction \push {lke}

\newcommand \getitem {\mathrm{item}} % REMOVE

\title{An alternative to hypercovers}
\author{Andrew W. Macpherson}
\begin{document}
\maketitle
\begin{abstract}

  I introduce a class of diagrams in a Grothendieck site called \emph{atlases} which can be used to study hyperdescent, and show that hypersheaves take atlases to limits using an indexed `nerve' construction that produces hypercovers from atlases. Atlases have the flexibility to be at the same time more explicit and more universal than hypercovers.

\end{abstract}

\tableofcontents

\section{Introduction}

A presheaf $F$ on a topological space $X$ is said to \emph{satisfy descent} along an open cover $\{U_i\subseteq V\}_{i:I}$ of an open set $V\subseteq X$ if its sections over $V$ can be computed by taking the limit over the \v{C}ech nerve
\begin{equation}
  \begin{tikzcd}
    \cdots \ar[r]\ar[r, shift left=1ex]\ar[r, shift right=1ex] & 
      \left( \coprod_{i:I}U_i \right) \times_X \left( \coprod_{i:I}U_i \right) 
        \ar[l, shift left=.5ex] \ar[l, shift right=.5ex]    
        \ar[r, shift left = .5ex] \ar[r, shift right = .5ex] & 
      \coprod_{i:I}U_i  \ar[l]
  \end{tikzcd}
  ( \longrightarrow X) 
\end{equation}
(for sheaves of sets, the first two terms, which appear here explicitly, are enough). It is called a \emph{sheaf} if it satisfies descent along all open covers. This perspective, popularised in \cite{SGA4}, is the starting point for topos theory and the modern approach to numerous flavours of cohomology theory.

For computing global sections  --- i.e.~cohomology --- it is useful to be able to restrict attention to covers belonging to a certain \emph{site} for $X$. For example, if $X$ is a manifold, then a convenient site is the poset $\opens^o(X)$ of open sets diffeomorphic to $\R^n$. Since these are exactly the contractible open sets, the classifying space of this poset has the same homotopy type as $X$. Hence, this site is useful for computing homotopy invariants of $X$, such as the theory of local systems on $X$.

Because contractible subsets of $X$ are not closed under intersection, the \v{C}ech nerve construction is not available for every cover in the site $\opens^o(X)$. The preceding definition of descent along open covers does not, therefore, translate easily into the new context. What we need is a generalisation in which a binary intersection of open subsets can be resolved with its own open cover, and again for the intersections of  \emph{those} open sets, and so on \emph{ad infinitum}. This is the intuition that the hypercovers of \cite[Exp.~V, 7.3]{SGA4} attempt to capture.

Today, the theory of hypercovers is woven into the fabric of nearly all flavours of homotopical sheaf theory (with the bulk of \cite[Chap.~6]{HTT} being a notable exception). They admit an elegant characterisation in terms of the model structure on simplicial presheaves \cite{DHI}. 

In practice, we must often construct hypercovers by conversion from more basic `naturally occurring' data. This process destroys finiteness and, perhaps, intuitiveness. Why not try instead to capture these `naturally occuring' data directly?
D.~Quillen is supposed to have said that only when he ``freed [himself] from the shackles of the simplicial way of thinking'' was he able to discover his Q-construction in algebraic $K$-theory \cite{Grayson:Quillen}.
The notion of \emph{atlas} I now define arose in the search for a similar liberty in the context of descent theory:

\begin{definition}[Atlas]
\label{atlas/definition}

  Let $T$ be an $\infty$-category equipped with a Grothendieck topology in the sense of \cite[Def.~6.2.2.1]{HTT}, $X:T$ an object. An $\infty$-functor $U:I\rightarrow T_{/X}$ is said to be an \emph{atlas} for $X$ if for any finite direct category $K$ and monotone map $\alpha:K\rightarrow I$, the set
  \[ 
    \{U_{\tilde{\alpha}} \rightarrow \lim_{k:K}U_{\alpha(k)}  \}_{\tilde\alpha:I_{/\alpha}}
  \] 
  is a covering in $\Presheaf(T)$, where $I_{/\alpha}$ is the overcategory of \cite[\S1.2.9]{HTT} (also known as the category of left cones over $K\rightarrow I$), and $U_{\tilde\alpha}$ stands for the value of $U\circ\tilde\alpha:K^\triangleleft\rightarrow T_{/X}$ on the cone point.

\end{definition}

\begin{definition}[Descent along atlases]
\label{atlas/descent/definition}

  An $\infty$-presheaf $F:T^\op\rightarrow\Space$ is said to \emph{satisfy descent along atlases} if each open $V\subseteq X$ and atlas $(I,U)$ of $V$ induces an equivalence of spaces
  \[
    F(V)\tilde\rightarrow \lim_{i:I}F(U_i). 
  \]
  (Here, of course, the limit in the $\infty$-category $\Space$ of spaces must be understood in the sense of $\infty$-category theory \cite[Def.~1.2.13.4]{HTT}.)
  
\end{definition}

The relevance of this definition is captured by the following statement:

\begin{conjecture}
\label{main/conjecture}

  The following conditions on an $\infty$-presheaf $F:T^\op\rightarrow\Space$ are equivalent:
  \begin{enumerate}
  
  \item $F$ satisfies descent along atlases.
  
  \item $F$ satisfies hyperdescent.
  
  \end{enumerate}
  In particular, the full subcategory of $\Presheaf_\infty(X)$ spanned by the functors that satisfy descent along atlases is a hypercomplete topos.

\end{conjecture}

The main application I have for this theory of atlases is a formulation of the universal property of (derived) geometry, which appears in a separate work \cite{Mac:DerivedGeometry}.
In the present paper, I prove a truncated version of Conjecture \ref{main/conjecture} that applies to topological spaces --- more precisely, to locales --- and is sufficient for the application of \emph{op.~cit}..

\begin{theorem}
\label{main/descent-theorem}

  Let $X$ be a locale with lattice of open sets $\opens(X)$, and let $F:\opens(X)^\op\rightarrow\Space$ be a hypercomplete $\infty$-sheaf.\footnote{More accurately, $F$ is a sheaf on the nerve of the poset $\opens(X)$.} Then $F$ satisfies descent along atlases indexed by posets.
  
\end{theorem}
\begin{proof}

  It is well-known that hypercompleteness means that $F$ takes hypercovers to limits. See \cite[\S3]{HAGI} (below Definition 3.4.8) for a proof in the context of simplicially enriched categories; to translate into the language of quasi-categories, use  \cite[Prop.~4.2.4.4]{HTT} and \cite[Rmk.~6.5.2.15]{HTT}.

  The data of a hypercover of $V\subseteq X$ can be equivalently formulated as a certain kind of diagram $J\slice\Delta^\op \rightarrow \opens(V)$ indexed by the total space $J$ of a left fibration over $\Delta^\op$, called its \emph{index diagram} \eqref{semi-representable/index-diagram}, which has a certain \emph{local lifting property} (Prop.~\ref{hypercover/fill-condition}).
  
  The result now follows from a \emph{nerve} construction which associates to each diagram $U:I\rightarrow\opens(V)$ of open sets of $V$, indexed by a poset $I$, a diagram 
  \begin{equation*}
    \begin{tikzcd}
      \smallint\nerve(I) \ar[r, "\epsilon"] \ar[d, "\LeftFibration" description] & I \\ \Delta^\op 
    \end{tikzcd}
  \end{equation*}
  such that $U\circ\epsilon:\smallint\nerve(I)\slice\Delta^\op \rightarrow \opens(V)$ has the local lifting properties of a hypercover if and only if $U$ is an atlas of $V$ (Theorem \ref{nerve/theorem}).
  Moreover, $\epsilon$ is $\infty$-cofinal in the sense of \cite[\S4.1]{HTT} (Proposition \ref{nerve/counit/is-cofinal}), whence a colimit over $I$ can be computed on its restriction to $\smallint\nerve(I)$ \cite[Prop.~4.1.1.8]{HTT}.
  Thus if $F$ is hypercomplete and $U$ is an atlas, then $F$ takes $(\smallint\nerve(I), U\circ\epsilon)$ --- and therefore also $(I,U)$ --- to a limit.
  \qedhere
  
\end{proof}

\paragraph{Acknowledgement} Thanks to Adeel Khan whose question exposed an error in an earlier draft of Definition \ref{atlas/definition}.

\section{Preliminaries}

This section contains some background information about posets, 1-categories, and $\infty$-categories, and discussion of coinitiality and initial objects in these contexts.

\begin{para}[$\infty$-categories and $n$-categories]
\label{category}

  Partially ordered sets and classical ($\Set$-enriched) categories are realised, via the nerve construction, as $0$- and $1$-categories inside the category of quasi-categories \cite[Ex.~2.3.4.3, Prop.~2.3.4.5]{HTT}.
  We do not make any linguistic distinction between a poset (resp.~classical category) and its nerve. (This only affects discussions where 0- or 1-categories interact with $\infty$-categories, i.e.~the proof of Proposition \ref{atlas/of-locale/criterion}.)
  The $\infty$-category of $n$-categories, for $n:\N\sqcup\{\infty\}$, is denoted $\Cat_n$.

\end{para}

\begin{para}[0-Truncation]
\label{category/truncation}

  The inclusion $\Cat_0\rightarrow \Cat_\infty$ has a right adjoint $h_0$ (see \cite[Rmk.~2.3.4.13]{HTT}, although this reference does not provide details).
  A realisation of the right adjoint is as follows: given a quasicategory $C$, $h_0C$ is the poset whose elements are the vertices of $C$ and whose relations are $x\leq y$ if and only if $C(x,y)$ is nonempty. 
  From this description it is also clear that the formation of the left cone (as defined in \cite[Not.~1.2.8.4]{HTT}) commutes with $h_0$.

\end{para}

\begin{definition}[0-coinitial]
\label{0-coinitial}

  A subset $K_0\subseteq K$ of a poset $K$ is \emph{0-coinitial} if every element of $K$ is bounded below by an element of $K_0$. 
  
  More generally, a functor $u:K_0\rightarrow K$ between two categories is called 0-coinitial if the image of $K_0$ in $h_0(K)$ is 0-coinitial. Equivalently, for every object $k:K$ there exists a $k_0:K_0$ and a morphism $uk_0\rightarrow k$.

\end{definition}

\begin{lemma}[Pushout of a cone along a coinitial map]
\label{util/cone/reduce-to-coinitial}

  Let $K$ be a poset, $K_0\subseteq K$ a $0$-coinitial subset. Then
  \[
    \begin{tikzcd}
      K_0 \ar[r] \ar[d] & K \ar[d] \\
      K_0^\triangleleft \ar[r] & K^\triangleleft
    \end{tikzcd}
  \]
  is a pushout in the category of posets.
  
\end{lemma}
\begin{proof}
  
  The statement for underlying sets is obvious, so we are just checking that all the order relations on $K^\triangleleft = K \sqcup \{e\}$ factorise as strings of relations that lift to  $K_0^\triangleleft$ and $K$.
  The only relations $i\leq j$ in $K^\triangleleft$ that do not lift to $K$ or $\{e\}$ are those for which $i=e$ is the cone point and $j\in K$. In this case, as $K_0$ is 0-coinitial there is some $i'\leq j$ with $i'\in K_0$, and the relation factorises as $i= e \leq i' \leq j$ with $e\leq i'$ a relation in $K_0^\triangleleft$.
  \qedhere
  
\end{proof}

\begin{definition}[Left 0-finite]
\label{left-0-finite}

  A 1-category $K$ is said to be \emph{left 0-finite} if any functor from $K$ into a finitely complete poset admits a limit cone.
  
  For example, this is the case for any $K$ admitting a 0-coinitial subset; in particular, when $K$ admits an initial object.
  
\end{definition}

\begin{para}[Finite intersections]
\label{finite-limit}
  
  Let $U:I\rightarrow J$ be a functor into a finitely complete poset $J$. 
  If $K$ is a left 0-finite category and $\alpha:K\rightarrow I$ is a functor, write 
  \[
    U_{\alpha} \defeq \lim_{k:K} U_{\alpha k} \quad \in \quad \opens(X)
  \]
  for the limit of $U$ in $\opens(X)$ over the diagram $\alpha$.
  Since $\opens(X)$ is a poset, this limit can be computed as an intersection
  \[
    U_\alpha = \bigcap_{k:K_0}U_{\alpha k}
  \]
  where $K_0\subseteq K$ is any finite 0-coinitial subset. 
  If in particular $K$ admits an initial object $e$, then of course $U_\alpha=U_{\alpha (e)}$.
  
\end{para}

\begin{eg}[Simplex boundary]
\label{finite-limit/example/simplex-boundary}

  The category of simplices of the simplex boundary $\smallint_{(+)}\partial\Delta_{(+)}^n$ (see \S\ref{sset} for notation) has a 0-coinitial subset 
  \[
    \{\Delta^{n-1} \stackrel{\sigma_j}{\hookrightarrow}\Delta^n\}_{j=0}^n
  \]
  comprising the facets of $\Delta^n$.
  In particular, $\int_{(+)}\partial\Delta_{(+)}^n$ is left $0$-finite.
  
  For any functor $U:I\rightarrow J$ and map $\tau:\smallint_{(+)}\partial\Delta^n_{(+)}\rightarrow I$, we calculate
  \[
    U_\tau \quad = \quad \bigcap_{j=0}^n U_{\tau\sigma_j}.
  \]
  
\end{eg}

Moving on to the case of cones on a category with initial object:

\begin{para}[Evaluation at initial object is left adjoint to diagonal]
\label{util/cone/initial-object/adjunction}

  Let $K$ be a category with initial object $e$. Evaluation on $e$ is a left adjoint to the constant functor 
  \[
    \constant:I \rightarrow \Fun(K,I), \quad i\mapsto \underline{i}
  \]
  with counit induced by applying $\Fun(-,I)$ to 
  \[
    [\underline{e}\rightarrow\id_K]:\Delta^1\times K\rightarrow K
  \]%
  resulting in a map $\Fun(K,I)\rightarrow \Fun(\Delta^1\times K,I)=\Fun(K,I)^{\Delta^1}$ that transforms $\constant\circ\evaluation_e$ into the identity.
  
  In particular, from the mapping space formula of the adjunction:
  \[
    \Fun(K,I)(\underline{i}, \phi) \cong I(i, \phi(e))
  \]
  for any $\phi:K\rightarrow I$. (It is natural to identify the left-hand side with the category of cones over $\phi$ with vertex $i$ \cite[67]{maclane}.)
  
\end{para}

\begin{lemma}
\label{util/cone/initial-object/pullback}

  Let $K:\Cat_1$ have an initial object $e$. The square
  \begin{equation*}
    \begin{tikzcd}
      \Fun(K,i\downarrow I) \ar[r] \ar[d, "\evaluation(e)"'] & \Fun(K,I) \ar[d, "\evaluation(e)"] \\
      i\downarrow I \ar[r] & I
    \end{tikzcd}
  \end{equation*}
  is a pullback in $\Cat_1$.
  
\end{lemma}
\begin{proof}

  The top-right element can be identified with the category of functors $\psi:\Delta^1\times K \rightarrow I$ with a trivialisation $\psi|_{0}\cong\underline{i}$, as follows:
  \begin{align*}
    \Fun(K,i\downarrow I) &\cong \Fun(K,\{i\}\times_II^{\Delta^1}) \\
    &\cong \{\underline{i}\} \times_{\Fun(K,I)} \Fun(K,I^{\Delta^1}) \\
    &\cong \{\underline{i}\} \times_{\Fun(K,I)} \Fun(\Delta^1\times K,I) \\
    &\cong \{\underline{i}\} \times_{\Fun(K,I)} \Fun(K,I)^{\Delta^1}.
  \end{align*}
  With this identification, the horizontal arrow is the target projection and the vertical is evaluation on $e$.
  
  Now we observe that both horizontal arrows are left fibrations, so it suffices to check that for each $\phi:K\rightarrow I$ evaluation at $\Delta^1\times\{e\}$ induces a bijection
  \[
    \left\{
      \begin{array}{c}
        \psi:\Delta^1\times K\rightarrow I \\
        \psi|_0 \cong \underline{i},\;
        \psi|_1 \cong \phi
      \end{array}
    \right\} 
    \cong I(i,\phi(e)).
  \]
  This follows from the adjunction $\evaluation(e)\dashv\constant$ \eqref{util/cone/initial-object/adjunction}.
  \qedhere

\end{proof}

\section{Atlases for locales}
\label{atlas}

If we restrict attention from arbitrary $\infty$-categories to posets, unsurprisingly we find some substantial simplifications to the general theory. 

\begin{para}[Diagram of opens sets]
\label{atlas/of-locale}

  Let $X$ be a locale with frame of open sets $\opens(X)$ \cite{Johnstone:Stone}.\footnote{If the reader prefers, he may instead let $X$ be a topological space without affecting the arguments. Descent theory is in any case mediated through the associated locale.} A \emph{diagram of open subsets} of $X$ is a monotone map of posets $U:I\rightarrow\opens(X)$. We often abbreviate these data as a tuple $(I,U$). We write $U_i\subseteq X$ for the open set associated to $i:I$ by $U$.

\end{para}

\begin{proposition}[Atlas for a locale]
\label{atlas/of-locale/criterion}
  
  Let $U:I\rightarrow \opens(X)$ be a diagram of open subsets of $X$ indexed by a poset $I$. The following conditions are equivalent:
  \begin{enumerate}
    
    \item $X=\bigcup_{i:I}U_i$, and for any $i,j:I$, $U_i\cap U_j=\bigcup_{k\leq i,j}U_k$.
    
    \item For any finite $J\subseteq I$,
    \begin{equation*}
      \coprod_{\genfrac{}{}{0pt}{2}{i:I}{i\leq j\;\forall j:J}}U_i 
        \twoheadrightarrow \bigcap_{j:J}U_j
    \end{equation*}
    is a covering.
    
    \item (The nerve of) $U$ is an atlas.
    
  \end{enumerate}
  
\end{proposition}
\begin{proof}

  $1\Leftrightarrow 2$ by induction on the cardinality of $J$.

  The criterion 2 is a special case of the definition of atlas; hence $3\Rightarrow2$. Conversely, suppose that $U:I\rightarrow\opens(X)$ satisfies 2, and let $K\rightarrow I$ be a finite diagram with $K$ some $\infty$-category. Then
  \[
    \opens(X)_{/\alpha} = \opens(X)_{/h_0\alpha}
  \]
  because $\opens(X)$ is 0-truncated and the formation $K\mapsto K^\triangleleft$ of the left cone commutes with truncation \eqref{category/truncation}.
  So, replacing $K$ with its truncation $h_0K$, we may assume that it is a finite poset.
  But then also
  \[
    \opens(X)_{/\alpha} = \opens(X)_{/\alpha\mid_{K_0}}
  \]
  where $K_0$ is $K$ regarded as a poset with the trivial ordering. 
  The atlas condition is now handled by 2.
  \qedhere
  
\end{proof}

The language of atlases is more flexible than that of hypercovers: many hypercovers of interest are obtained by conversion from a naturally arising diagram which may itself already be an atlas.

\begin{eg}
\label{atlas/example/basic}

  Let $X=U\cup V$ be a topological space expressed as a union of two open sets.
  The diagram
  \begin{equation*}
    \begin{tikzcd}
      U\cap V \ar[r] \ar[d] & V \\ U
    \end{tikzcd}
  \end{equation*}
  in $\opens(X)$ is an atlas for $X$.
  
  Suppose now we have a further decomposition $U\cap V= A\cup B$ with $A\cap B=\emptyset$. Then the diagram
  \begin{equation*}
    \begin{tikzcd}
      & A \ar[dr, bend left] \ar[dd] \\
      B \ar[rr, crossing over] \ar[dr, bend right] && U \\
      & V
    \end{tikzcd}
  \end{equation*}
  is an atlas for $X$. Notice that the index poset for this atlas has the weak homotopy type of $S^1$. The reader can no doubt imagine a way to realise this diagram as an atlas of contractible open sets in the case $X=S^1$.

\end{eg}

\begin{eg}[Basis]
  
  A basis $I\subseteq \opens(X)$ for the topology of $X$ is an atlas for $X$: condition 2 of Proposition \ref{atlas/of-locale/criterion} follows easily from the definition of a basis.

\end{eg}

\begin{eg}[Schemes]

  A locally ringed space $X$ is a scheme if and only if the inclusion $\opens^\mathrm{aff}(X)\rightarrow\opens(X)$ of the poset of open immersions from affine schemes is an atlas for $X$.
  Note that this poset is closed under finite intersections (resp.~binary intersections) if and only if $X$ is an affine scheme (resp.~has affine diagonal). Thus, the Cech nerve construction is not always available to us within this site.
  
\end{eg}

\begin{eg}[Manifolds]

  A paracompact Hausdorff space $X$ is a topological manifold if and only if the preorder $\opens^o(X)$ of open immersions $\R^n\hookrightarrow X$ is an atlas for $X$.
  This preorder is never closed under binary intersections (unless $X$ is a point). This example, or its $C^\infty$ analogue, was what first motivated me to formulate the notion of atlas.
  
\end{eg}

Atlases give us an easy way to describe certain counterexamples to hypercompleteness:

\begin{eg}[Hilbert cube]

  Consider the Hilbert cube $Q=[0,1]^\N$ as in \cite[Ex.\ 6.5.4.8]{HTT}. 
  The set of open subsets homeomorphic to $Q\times[0,1)$ forms a base for the topology; in particular, it is an atlas. 
  Borel-Moore homology defines a sheaf (by excision) whose restriction to this atlas is zero, but whose global sections are nonzero; hence, by Theorem \ref{main/descent-theorem} it cannot be hypercomplete.

\end{eg}

\section{Simplices}
\label{sset}

Begin with some preliminary remarks on simplicial sets.

\begin{para}[Simplex categories]

  As usual, $\Delta$ is the geometric simplex category of inhabited totally ordered sets, while $\Delta_+$ is the subcategory consisting of injective maps.
  By \cite[Lemma 6.5.3.7]{HTT}, the inclusion $\Delta_+\rightarrow\Delta$ is $\infty$-coinitial, i.e.~for each $n$, $\Delta_+\downarrow_{\Delta}\Delta^n$ is weakly contractible.
  
\end{para}

\begin{para}[Simplices of a simplicial set]
\label{sset/simplices}

  A simplicial set $F$ can be realised as a left fibration in sets by integrating over $\Delta^\op$.
  We denote this fibration by 
  \begin{equation*}
  %\label{sset/simplices/type}
    \smallint F \rightarrow \Delta^\op
  \end{equation*}
  or, more briefly, $\smallint F\slice\Delta^\op$.
  It is called the \emph{category of simplices} of $F$.
  
  Similarly, a semisimplicial set $G$ can be realised as a left fibration $\int_+G \rightarrow \Delta_+^\op$. This is called the category of \emph{nondegenerate simplices} of $G$.  
  
  Let $\senv G$ be the \emph{simplicial envelope} of $G$ --- that is, simplicial set obtained from $G$ by left Kan extension along $\Delta_+\subset\Delta$. As an extension it comes equipped with a canonical functor $\int_+G\rightarrow\int \senv G$.
  
\end{para}

\begin{eg}[Simplex]
\label{sset/example/simplex}

  We write $\Delta^n_+$ for the semi-simplicial set represented by an ordered set with $n+1$ elements.
  The category $\int_+\Delta_+^n$ of nondegenerate simplices is the opposite of the poset of inhabited subsets of $[n+1]$ --- in particular, it is finite. 
  
  The simplicial envelope of $\Delta_+^n$ --- the simplicial set represented by the same ordered set considered as an object of $\Delta$ --- is written $\Delta^n$.
  Its category of simplices $\int\Delta^n$ is infinite, but taking the image of a map $\Delta^k\rightarrow\Delta^n$ defines a coreflector onto the finite subset $\int_+\Delta_+^n$ of nondegenerate simplices. In particular, it is left finite in that limits over $\int\Delta^n$ are equivalent to limits over a finite category --- cf.~\eqref{finite-limit}.
  
\end{eg}

\begin{eg}[Simplex boundary]
\label{sset/example/simplex-boundary}

  Write $\partial\Delta^n_+$ for the semi-simplicial set representing the boundary of the $n$-simplex, and $\partial\Delta^n$ for the associated simplicial set. 
  The category 
  \[\textstyle
    \smallint_+\partial\Delta_+^n \quad = \quad \smallint_+\Delta^n \setminus \{\Delta^n\stackrel{\mathrm{id}}{\rightarrow}\Delta^n\}
  \]
  of nondegenerate simplices of $\partial\Delta^n$ is the poset of inhabited proper subsets of $[k+1]$.
  In particular, 
  \[ \textstyle
    \smallint_+\Delta_+^n = \left(\smallint_+\partial\Delta_+^n\right)^\triangleleft
  \]
  is a categorical left cone over the category of nondegenerate simplices of $\partial\Delta_+^n$.
  
  Similarly, $\smallint\partial\Delta^n=\smallint\Delta^n\times_{\smallint_+\Delta^n_+}\smallint_+\partial\Delta^n_+$ is the category of simplices equipped with a non-surjective map to $\Delta^n$. 
  It is a sieve in $\smallint\Delta^n$.
  It contains $\smallint_+\partial\Delta_+^n$ as a coreflective subcategory and is therefore left finite.   

\end{eg}

\begin{eg}[Mixed category of simplices of a simplex]
\label{sset/example/mixed}

  The square
  \begin{equation*}
    \begin{tikzcd}
      \int_+\partial\Delta_+^n \ar[r] \ar[d] &
      \int_+\Delta_+^n \ar[d] \\
      \int\partial\Delta^n \ar[r] &
      \int\Delta^n
    \end{tikzcd}
  \end{equation*}
  of fully faithful functors isn't quite a pushout in $\Cat$: rather, the pushout is the full subcategory $\int\partial\Delta^n\cup\int_+\Delta_+^n$ of $\int\Delta^n$ spanned by simplices which are either nondegenerate or factor through the boundary.
  This subcategory is coreflective with coreflector fixing the boundary.
  Precisely, the value of the coreflector on $\sigma:\Delta^k\rightarrow\Delta^n$ is either $\sigma$ itself, if the image of $\sigma$ is contained in the boundary, or $\mathrm{Im}(\sigma)$ otherwise.
  
\end{eg}

\section{Index diagrams}
\label{index-diagram}

Hypercoverings of $X$ are defined in \cite[Exp.\ V, \S7.3]{SGA4}, \cite[Def.~4.2]{DHI}, \cite[Def.~3.4.8]{HAGI} as certain simplicial objects in the category of presheaves on $\opens(X)$. 
Some manipulation is required to convert these into diagrams in $\opens(X)$ itself.

In this section, we discuss the exchange of a semi-representable presheaf with a mapping indexing its connected components.

\begin{comment}
\begin{remark}[Hypercoverings in \cite{HTT}]
\label{hypercover/htt}

  The notion of hypercovering appearing in \cite[\S6.5.3]{HTT} is a bit different, as it takes values in a topos without making reference to any explicit site. 
  A presheaf is a hypercomplete sheaf if and only if it satisfies hyperdescent with respect to all hypercovers in the sheaf topos \cite[Thm.~6.5.3.12]{HTT}. 
  For the present purposes, it is essential to be able to work with simplicial semi-representable presheaves over $\mathcal U(X)$.
  
\end{remark}
\end{comment}

\begin{definition}[Semi-representable presheaves]
\label{semi-representable/definition}

  Let $C$ be a poset. Recall from \cite[Exp.~V,7.3]{SGA4} that a presheaf of sets on $C$ is said to be \emph{semi-representable} if it can be represented as a coproduct of representables.
  The full subcategory of the 1-category $\PSh_1(C)$ of presheaves of sets on $C$ spanned by the semi-representable objects is denoted $\SemiRepresentable(C)$.

\end{definition}

\begin{para}[Set indexed objects]
\label{indexed-object/definition}

  An $S$-indexed element of $C$, where $S:\Set$, is a map $S\rightarrow C$. 
  We write 
  \[
    C^S = \Fun(S,C)
  \]
  for the poset of $S$-indexed elements of $C$.
  Applying a contravariant Grothendieck construction, the 1-category of all set-indexed elements of $C$ is defined 
  \[
    \int_SC^S=\int_{S:\Set}C^S \quad \stackrel{\idx_C}{\rightarrow} \quad \Set
  \]
  as Cartesian fibration over $\Set$.
  Integrating the projection $\idx$ (which takes an \emph{indexed object} to its \emph{indexing set}) yields a left (co-Cartesian) fibration 
  \begin{equation}
  \label{indexed-object/index-set/universal}
    \widetilde{\idx}_C \defeq \int_{\int_SC^S}\idx_C\rightarrow\int_SC^S,
  \end{equation}
  the \emph{universal index set} of set-indexed objects of $C$.
  
\end{para}

\begin{para}
\label{indexed-object/lke}

  The functor $\push_{-/\point}:\int_{S:\Set}\Fun(S,C) \rightarrow \SemiRepresentable(C)$ is constructed as follows:
  \begin{itemize}
    \item 
      An object $X:S\rightarrow C$ goes to the coproduct presheaf $\coprod_{s:S}X_s:\Presheaf_1(C)$;
    
    \item 
      A morphism
        \begin{equation*}
          \begin{tikzcd}[column sep = small]
            S_0 \ar[rr, "f"] \ar[dr, "X_0"'] & \ar[d, phantom, "\Rightarrow"] & S_1 \ar[dl, "X_1"] \\
            & C
          \end{tikzcd}
        \end{equation*}
        in $\int_{S:\Set}C^S$ gets sent to the map
        \begin{align*}
          \coprod_{s:S_0}X_{0,s} 
            & \alignhack{\cong}{\rightarrow} \coprod_{s:S_1}(\push_fX_0)_s \\
          & \alignhack{\cong}{\rightarrow} \coprod_{s:S_1}\coprod_{t:f^{-1}s} X_{0,t} \\
          & \rightarrow \coprod_{s:S_1}X_{1,s}
        \end{align*}
        obtained from the universal property of the left Kan extension $\push_f(X_0)$ of $X_0$ along $f$.
  
  \end{itemize}
  Compatibility with composition follows from uniqueness of the map from the left Kan extension.
  
\end{para}

\begin{proposition}
\label{indexed-object/lke/is-equivalence}

  The groupoid of representations of a semi-representable presheaf as a coproduct of representables is contractible.
  The natural functor 
  \begin{equation*}
    \push_{-/\point} : \int_{S:\Set}C^S \;\tilde\rightarrow\; \SemiRepresentable(C) 
  \end{equation*}
  from the 1-category of set-indexed objects of $C$ to the 1-category of semi-representable presheaves on $C$ is an equivalence of 1-categories.

\end{proposition}

\begin{proof}

  The functor $\push_{-/\point}$ is essentially surjective by the definition of $\SemiRepresentable(C)$. Hence, we just have to show that it is fully faithful. If $(I,X)$ and $(J,Y)$ are two set-indexed objects, then maps from $X$ to $Y$ are defined as 
  \begin{equation*}
    \left(\int_{S:\Set}C^S\right)((I,X),(J,Y)) \cong \coprod_{f:J^I}\prod_{i:I} C(X_i, Y_{f(i)}).
  \end{equation*}
  Now, the $\#I$ projections
  \begin{equation*}
    \begin{tikzcd}
    \coprod_{f:J^I}\prod_{i:I}C(X_i, Y_{f(i)}) \ar[r] \ar[d, "\mathrm{ev}_i"] 
    & J^I \ar[d, "\mathrm{ev}_i"] \\
    \coprod_{j:J} C(X_i, Y_j) \ar[r] & J
    \end{tikzcd}
  \end{equation*}
  identify this further with $\prod_{i:I}\coprod_{j:J}C(X_i, Y_j)$.
  (This is an application of the Grothendieck construction's commutation with external products; since the objects here are sets, it can be seen by direct computation.)
  
  That $\push$ is fully faithful is now visible from the identifications
  \begin{align*}
    \prod_i\coprod_jC(X_i, Y_j)
      &\cong \prod_i \SemiRepresentable(C) \left(X_i,\coprod_jY_j\right) \\
    &\cong \SemiRepresentable(C) \left(\coprod_{i:I}X_i, \coprod_{j:J}Y_j\right)
  \end{align*}
  with the first equality following because $\SemiRepresentable(C)(X,-)$ commutes with coproducts for representable $X$.
  \qedhere
  
\end{proof}

\begin{para}[Tensoring over $\Set$]
\label{indexed-object/tensoring}

  It follows from Proposition \ref{indexed-object/lke/is-equivalence} that the coproduct over any set $(I_j,X_j)_{j:J}$ of objects of $\int_{S:\Set}C^S$ is exhibited by the obvious morphisms
  \[
    (I_j, X_j) \rightarrow \left(\coprod_{j:J}I_j, \coprod_{j:J}X_j\right).
  \]
  In particular, this defines a formula for a (coproduct-preserving) tensoring of $\int_{S:\Set}C^S$ over $\Set$.
  
  If $C$ has a maximal element, as is the case for the example of interest $C=\opens(X)$, then we can alternatively realise this tensoring as Cartesian product inside $\SemiRepresentable(C)$ by embedding $\Set$ as the full subcategory spanned by formal coproducts of this maximal element.
  This perspective will take priority when considering the parametrised version below \eqref{indexed-object/family/tensoring}.
  
\end{para}

\begin{para}[Local isomorphisms]
\label{semi-representable/local-isomorphism}

  Via totalization $\push_{-/\point}$, a commuting triangle
  \begin{equation*}
    \begin{tikzcd}[column sep = tiny]
      S_0 \ar[rr, "f"] \ar[dr, "X_0"'] & & S_1 \ar[dl, "X_1"] \\
      & C
    \end{tikzcd}
  \end{equation*}
  induces a morphism of semi-representable presheaves. A morphism in $\SemiRepresentable(C)$ induced in this way is called a \emph{local isomorphism}. In other words, local isomorphisms are the morphisms which are Cartesian for the forgetful functor $\idx:\SemiRepresentable(C)\rightarrow \Set$.
  
  The Cartesian property means that for any morphism $F_0\rightarrow F_1$ in $\SemiRepresentable(C)$ factorises uniquely into a composite of a map with \emph{fixed index} --- induced by a morphism (i.e.~inequality) in $\Fun(\idx(F_0), C)$ --- followed by a local isomorphism over $\idx(F_0)\rightarrow\idx(F_1)$.
  
\end{para}

\begin{remark}[Terminology]  
  
  The intuition behind this terminology is as follows: a morphism $\phi:K\rightarrow L$ in $\SemiRepresentable(C)$ is called a local isomorphism if for each connected component $K'\subseteq K$, the restriction of $\phi$ to $K'$ exhibits it as a connected component of $L$. This concept extends term-wise to $s\SemiRepresentable(C)$.
  
  Beware that this notion of local isomorphism has nothing to do with any auxiliary site structure that may exist on $C$ (such as we have in the case $C=\opens(X)$): it refers to locality only in the sense of formal coproducts in $\SemiRepresentable(C)$.
  
\end{remark}

%%%%%%%%%%%%%%%%%%%%%%%%%%%%%%%%%%%%%%%%%%%%%%%%%%%%%%%%%%%%%%%%%%%%%%%%%%%%%%

\section{Index diagrams in families}

In this section we study simplicial objects in $\SemiRepresentable(C)$ and, building on \S\ref{index-diagram} exchange them with a mappings from a left fibration over $\Delta^\op$.

\begin{proposition}[Families of set indexed objects]
\label{indexed-object/family}
  
  The 1-category of set-indexed objects of $C$ classifies diagrams
  \begin{equation*}
    \begin{tikzcd}
      I \ar[r] \ar[d, "\LeftFibration" description] & C \\ K
    \end{tikzcd}
  \end{equation*}
  where $I\rightarrow K$ is a left fibration in sets. That is, for any 1-category $K$ there is an equivalence of categories
  \begin{equation}
    \push_{-/K}:
    \left\{
      \begin{tikzcd}
        \cdot \ar[r] \ar[d, "\LeftFibration" description] & C  \\
        K
      \end{tikzcd}
    \right\}
    {\tilde\longrightarrow}
    \Fun\left(K,\int_SC^S\right).
  \end{equation}
\end{proposition}
\begin{proof}
  Beginning from the alias 
  \begin{equation*}
    \int_SC^S = \Set\downarrow_{\Cat_0}\{C\} 
  \end{equation*}
  we find
  \begin{equation*}
    \begin{array}{@{\vspace{1ex}}rclr}
      
      \Fun\left(K,\int_SC^S\right) &\cong& 
      \left\{
      \begin{tikzcd}
        \ar[dr, phantom, "\Downarrow" very near end] & 
          \Set \ar[d, hook] \\
        K \ar[ur] \ar[r, "{\underline{C}}"'] &  
          \Cat_0
      \end{tikzcd}
      \right\} & \\
      
      & \cong &
      \left\{
      \begin{tikzcd}[column sep = small]
        \cdot \ar[rr] \ar[dr, "\LeftFibration" description] && 
          K \times C \ar[dl, "\projection_K"] \\
        & K
      \end{tikzcd}
      \right\} &
      \text{via }\int_K \\
      
      & \cong &
      \left\{
      \begin{tikzcd}
        \cdot \ar[r] \ar[d, "\LeftFibration" description] & C  \\
        K
      \end{tikzcd}
      \right\}
      
    \end{array}
  \end{equation*}
  where the expressions in braces stand for categories of diagrams of the specified shape with restrictions indicated by the labels, that is:
  \begin{itemize} 
    \item named objects $(K, C,\ldots)$ and arrows ($\underline{C}$,$\projection_K$) are equipped with identifications with their eponym;
    \item arrows marked $\LeftFibration$ are constrained to the category of left fibrations in sets.
    \qedhere
  \end{itemize}
  
\end{proof}

\begin{para}[Tensoring over $\Fun(K,\Set)$]
\label{indexed-object/family/tensoring}

  Let $K:\Cat_1$. Taking the tensoring of $\SemiRepresentable(C)$ over $\Set$ \eqref{indexed-object/tensoring} pointwise, $\Fun(K, \SemiRepresentable(C))$ is tensored over $\Fun(K,\Set)$.
  
  Passing through the equivalence of Proposition \ref{indexed-object/family}, we find the formula
  \begin{equation}
  \label{indexed-object/family/tensoring/formula}
    F \otimes 
    \left(
    \begin{tikzcd}
      E \ar[r] \ar[d] & C  \\
      K
    \end{tikzcd}
    \right)
    \quad
    =
    \quad
    \left(
    \begin{tikzcd}
      \int_K F\times_K E \ar[r] \ar[d] & C  \\
      K
    \end{tikzcd}
    \right)
  \end{equation}
  for tensoring over $F:\Fun(K,\Set)$. This is easiest to see by returning to our assumption \eqref{indexed-object/tensoring} that $C$ has a greatest element so that the tensoring can be calculated as a Cartesian product in the diagram category.

\end{para}

\begin{para}[Simplicial semi-representable presheaves]
\label{semi-representable/index-diagram}

  Combining Propositions \ref{indexed-object/lke/is-equivalence} and \ref{indexed-object/family}, a simplicial object $U_\bullet$ of $\SemiRepresentable(C)$ yields a simplicial object of $\int_{S}C^S$ and hence a diagram
  \begin{equation*}
    \begin{tikzcd}
      \idx_U \ar[d] \ar[r] \ar[rr, bend left]
        \ar[dr, phantom, "\lrcorner", very near start] & 
        \widetilde\idx_C \ar[d] \ar[r] & C \\
      \Delta^\op \ar[r, "U"] & \smallint_S C^S
    \end{tikzcd}
  \end{equation*}
  where $\idx_U\rightarrow\Delta^\op$ is a left fibration.
  The category $\idx_U=\int_{\Delta^\op}(\idx\circ U)$ is called the \emph{index category} or \emph{category of indices} of the object $U_\bullet$, and the data $\getitem_U:\idx_U\rightarrow C$ is called the \emph{index diagram} of $U_\bullet$.
  
\end{para}

\begin{para}[Lifting property for simplicial objects]
\label{lifting-property/for-sset}

  Let $K_\bullet\subseteq L_\bullet$ be an inclusion of simplicial sets. A lifting for the diagram 
  \begin{equation}
    \begin{tikzcd}
      \smallint K \ar[r, "\sigma"] \ar[d] & \smallint J \ar[d] \ar[r, "U"] & C \\
      \smallint L \ar[r] \ar[ur, dotted] & \Delta^\op
    \end{tikzcd}
  \end{equation}
  is, removing the integrals throughout, equivalent to the data of an extension
  \begin{equation}
  \label{lifting-property/dhi/local-iso}
    \begin{tikzcd}
      \push_{\smallint K\slice\Delta^\op}(U\circ\sigma) \ar[r, "\sigma"] 
        \ar[d] & \push_{\smallint J\slice\Delta^\op}U \\
      \cdot \ar[ur, "\texttt{local-iso}" description, dotted]
    \end{tikzcd}
  \end{equation}
  where the blank is something with index category $\smallint \Delta^n\slice\Delta^\op$, the vertical arrow covers the inclusion $\partial\Delta^n\subset\Delta^n$, and, as indicated, the extension is required to be a local isomorphism in the sense of \ref{semi-representable/local-isomorphism}. (Note that this implies that the vertical arrow is also a local isomorphism.)
  Hence, remaining in the category of left fibrations handles the commutativity of the lower triangle in lifts of diagrams like \eqref{lifting-property/diagram}.

\end{para}

\begin{example}[Representable objects]

  Let $V:C$. The index diagram of the (simplicially constant) simplicial presheaf on $C$ represented by $V$ is
  \[
    \begin{tikzcd}
      \Delta^\op \ar[r, "\underline{V}"] \ar[d, equals] & C \\
      \Delta^\op
    \end{tikzcd}
  \]
  
\end{example}

\begin{para}[Tensor-representable objects]
\label{semi-representable/tensor-representable}

  Let $K_\bullet$ be a simplicial set and $V:C$. The index diagram of $V\times K_\bullet$, where this is computed with respect to the tensoring of $\SemiRepresentable(C)$ over $s\Set$, has the form
  \[
    \begin{tikzcd}
      \smallint K \ar[r, "\underline{V}"] \ar[d] & C \\ 
      \Delta^\op
    \end{tikzcd}
  \]
  If $U:\smallint K\rightarrow C$ is another diagram, then morphisms $K_\bullet\otimes V\rightarrow U$ are the same as cones over $U$ with vertex $V$. In particular, if $U$ admits a limit in $C$, then it admits a final object in the category of maps from tensor-representable objects.
  
\end{para}

%%%%%%%%%%%%%%%%%%%%%%%%%%%%%%%%%%%%%%%%%%%%%%%%%%%%%%%%%%%%%%%%%%%%%%%%%%%%%%

\section{Local lifting conditions}
\label{lifting-property}

Throughout this section and for the rest of the paper, $X$ will denote a locale with lattice of open subsets $\opens(X)$.

\begin{definition}[Local lifting conditions]
\label{lifting-property/definition}

  A \emph{local lifting problem} on $X$ is a diagram in $\Cat_1$ of the form
  \begin{equation}
  \label{lifting-property/diagram}
    \begin{tikzcd}
      K \ar[r, "\sigma"] \ar[d] & I \ar[d] \ar[r, "U"] & \opens(X) \\
      L \ar[r] & J
    \end{tikzcd}
  \end{equation}
  where, in the examples of interest:
  \begin{itemize}
    
    \item $K\rightarrow L$ is a fully faithful functor; especially, but not always, equivalent to $K\subset K^\triangleleft$.
    
    \item Either $J=\Delta^\op$ or $\Delta_+^\op$, in which case the maps from $K$, $L$, and $I$ are required to be left fibrations, or $J=\point$.
  
  \end{itemize}
  We say that this problem has a \emph{solution}, or that the square \emph{has the local lifting property} (local l.p.), if $U_\sigma$ is covered by sets of the form $U_\tau$, where $\tau$ ranges over completionss
  \[
    \begin{tikzcd}
      K \ar[r, "\sigma"] \ar[d] & I \ar[d] \\
      L \ar[r] \ar[ur, dotted, "\tau"] & J
    \end{tikzcd}
  \]
  of the square.
  
  If $K\subseteq L$ is a functor of 1-categories and $U:I\slice J\rightarrow\opens(X)$ is a diagram, and moreover any local lifting problem \eqref{lifting-property/diagram} has a solution, we say that $K\slice L$ has the \emph{local left lifting property} or \emph{local l.l.p.}~against $(I\slice J,U)$, or that $(I\slice J,U)$ has the \emph{local right lifting property} or \emph{local r.l.p.}~against $K\slice L$.
  
\end{definition}

\begin{prop}
\label{lifting-property/closure-properties}
  Let $U:I\slice J\rightarrow\opens(X)$ be a diagram of open sets. The set of morphisms $K\slice L$ having the local l.l.p.~against $(I,U)$ has the following closure properties:

  \begin{enumerate}
  
    \item It is stable under composition, pushout, and retract. 
    
    \item Suppose that $J=\point$, $L\rightarrow L'$ under $K$, and $K\slice L'$ has the local l.l.p.~against $(I,U)$. Then $K\slice L$ has the local l.l.p.~against $(I,U)$.
    
  \end{enumerate}

\end{prop}
\begin{proof}

  The argument for 1 is the same as for any class of arrows defined by a right lifting property. For 2, any lifting problem \eqref{lifting-property/diagram} can be factorised
  \[
    \begin{tikzcd}
      K \ar[r, equals] \ar[d] & K \ar[r] \ar[d] & I \\
      L \ar[r] & L'
    \end{tikzcd}
  \]
  and so a lift for $K\slice L'$ yields a lift for $K\slice L$.
  \qedhere

\end{proof}

\begin{corollary}[Local lifting property reduces to a coinitial subset]
\label{lifting-property/reduce-to-coinitial}

  Let $K$ be a finite poset, $K_0\subseteq K$ a 0-coinitial subset. 
  Let $U:I\rightarrow\opens(X)$ be a poset-indexed diagram. 
  If $K_0\subset K_0^\triangleleft$ has the local l.l.p.~against $U$, then so does $K\subset K^\triangleleft$.
  
\end{corollary}
\begin{proof}

  By Lemma \ref{util/cone/reduce-to-coinitial}, $K\slice K^\triangleleft$ is a pushout of $K_0\slice K_0^\triangleleft$, and hence inherits the local l.l.p.~by Proposition \ref{lifting-property/closure-properties}-1. \qedhere

\end{proof}

\begin{corollary}[Local lifting properties for simplices versus semi-simplices]
\label{lifting-property/sset-to-ssset}

  The local r.l.p.~for $\int_+\partial\Delta_+^n\slice \int_+\Delta_+^n$ and for $\int\partial\Delta^n\slice \int\Delta$ are equivalent.
  
\end{corollary}
\begin{proof}

  Suppose $U$ has the local r.l.p.~against $\int_+\partial\Delta_+^n\slice \int_+\Delta_+^n$. By Proposition \ref{lifting-property/closure-properties}-1, it has the local r.l.p.~against $\int\partial\Delta^n\slice \int_+\Delta^n_+ \cup \int\partial\Delta^n$. By Example \ref{sset/example/mixed}, this is a retract of $\int\partial\Delta^n\slice \int\Delta$ fixing the boundary, and so by Proposition \ref{lifting-property/closure-properties}-2 we are done.

  Conversely, $\smallint\partial\Delta_+^n\slice\smallint\Delta_+^n$ is a retract of $\smallint\partial\Delta^n\slice\smallint\Delta^n$.
  \qedhere

\end{proof}

\begin{eg}
\label{lifting-property/example/atlas}

  The condition \ref{atlas/of-locale/criterion}-2 for a diagram $(I,U)$ to be an atlas is the local r.l.p.~for inclusions $K \slice K^\triangleleft$ where $K$ is a finite set.
  That is, these morphisms have the local l.l.p.~with respect to any atlas (considered over $J=\point$).
  In particular, the covering condition is the local r.l.p.~against $\int_+\partial\Delta^0 \slice \int_+\Delta^0$, and condition 1 is the local r.l.p.~for $\int_+\partial\Delta^1 \slice \int_+\Delta^1$.
  
\end{eg}

\begin{prop}[Atlases via lifting properties]
\label{lifting-property/atlas}

  The condition for a diagram $U:I\rightarrow\opens(X)$ to be an atlas is equivalent to any of the following local right lifting properties:
  \begin{enumerate}
  
    \item $K\slice K^\triangleleft$ for $K=\emptyset$ and $K=\{0,1\}$.
    
    \item $K\slice K^\triangleleft$ for all finite sets $K$.
    
    \item $\int_+\partial\Delta_+^n\slice \int_+\Delta_+^n$ for $n:\{0,1\}$.
    
    \item $\int_+\partial\Delta_+^n\slice \int_+\Delta_+^n$ for all $n:\N$.
    
    \item $\int\partial\Delta^n\slice \int\Delta^n$ for $n:\{0,1\}$.
    
    \item $\int\partial\Delta^n\slice \int\Delta^n$ for all $n:\N$.
    
  \end{enumerate}
  
\end{prop}
\begin{proof}

  As already observed (cf.~Ex.~\ref{lifting-property/example/atlas}), 1 and 2 correspond exactly to conditions 1 and 2 of \ref{atlas/of-locale/criterion}, respectively.
  
  The implication 2 $\Rightarrow$ 4 is an application of Corollary \ref{lifting-property/reduce-to-coinitial} and Example \ref{finite-limit/example/simplex-boundary}. Meanwhile 4 $\Rightarrow$ 3 because 3 is a specialisation of 4, and 1 $\Leftrightarrow$ 3 because these are equivalent categories.

  The equivalences 3 $\Leftrightarrow$ 5 and 4 $\Leftrightarrow$ 6 follow from Corollary \ref{lifting-property/sset-to-ssset}.
  \qedhere
  
\end{proof}

\section{Hypercovers via lifting conditions}
\label{hypercover}

Hypercovers are defined in terms of \emph{local lifting} conditions in the category of simplicial presheaves \cite[\S3]{DHI}. Via the construction of \ref{semi-representable/index-diagram}, these translate almost directly into local lifting conditions for diagrams, i.e.~in the sense of Definition \ref{lifting-property/definition}.
The key difference is that the lifting conditions of \emph{op.~cit}.~restrict attention to \emph{tensor-representable objects} \eqref{semi-representable/tensor-representable}, while our filling conditions are restricted to \emph{local isomorphisms} \eqref{semi-representable/local-isomorphism}. In this section, we define and compare these variants.

Denote by $\SemiRepresentable(X)$ the category of semi-representable presheaves on $\opens(X)$, and by $s\SemiRepresentable(X)$ its category of simplicial objects.

\begin{proposition}
\label{hypercover/fill-condition}

  The local lifting property for a diagram
  \begin{equation}
  \label{lifting-property/simplex}
    \begin{tikzcd}
      \int\partial\Delta^n  \ar[d] \ar[r, "\sigma"] & 
        \smallint K \ar[d, "\chi"] \ar[r, "U"] & \opens(X) \\
      \int\Delta^n \ar[r] & \Delta^\op
    \end{tikzcd}
  \end{equation}  
  is equivalent to the square
  \begin{equation}
  \label{lifting-property/dhi}
    \begin{tikzcd}
      \partial\Delta^n \otimes U_\sigma  \ar[d] \ar[r, "\sigma"] & 
        \push_{\smallint K\slice\Delta^\op} U \ar[d] \\
      \Delta^n \otimes U_\sigma \ar[r] & \point
    \end{tikzcd}
  \end{equation}
  admitting local liftings in the sense of \cite[\S3]{DHI}, where as above $U_\sigma$ denotes the limit of $U\circ \sigma$.
  
  In particular, an object $U_\bullet:\SemiRepresentable(X)$ is a hypercover if and only if its index diagram has the local r.l.p.~against $\smallint\partial\Delta^n\slice \smallint\Delta^n \slice \Delta^\op$ for all $n:\N$.
  
\end{proposition}
\begin{proof}
    
  We will use the same strategy for representing categories of diagrams as in the proof of Proposition \ref{indexed-object/family}. 
  Moreover, we consider \eqref{lifting-property/simplex} in the category of fibrations over $\Delta^\op$ as in \ref{lifting-property/for-sset}; hence we may disregard the lower triangle when constructing fillers for \eqref{lifting-property/simplex}, \eqref{lifting-property/dhi}.
  
  By Lemma \ref{util/fibration/power}, the fibre of
  \begin{equation*}
    \left\{
      \begin{tikzcd}
        \partial\Delta^n \otimes V  \ar[d] \ar[r] &
          \push_{\smallint K\slice\Delta^\op} U \ar[d, equals]  \\
        \Delta^n \otimes V \ar[r] & \push_{\smallint K\slice\Delta^\op} U
      \end{tikzcd}
    \right\}
    \quad \stackrel{\idx}{\rightarrow} \quad
    \left\{
      \begin{tikzcd}
        \smallint\partial\Delta^n   \ar[d] \ar[r, "\partial\tau"] &
          \smallint K \ar[d, equals]  \\
        \smallint\Delta^n  \ar[r, "\tau"] & \smallint K
      \end{tikzcd}
    \right\},
  \end{equation*}    
  where $\tau$ is a variable, is a truth value equal to 
  \begin{equation*}
    \left\{
      \begin{array}{rcll}
        V & \leq & U\circ\partial\tau & \text{in } \Fun(\smallint\partial\Delta^n, \opens(X)) \\
        V & \leq & U\circ\tau & \text{in }\Fun(\smallint\Delta^n,\opens(X))
      \end{array}
    \right\}.
  \end{equation*} 
  If true, the inequalities factorise the above as: 
  \begin{equation*}
    \begin{tikzcd}[column sep = large]
      \partial\Delta^n \otimes V  \ar[d] \ar[r, "\texttt{fixed-idx}"] &
        \push_{\smallint \partial\Delta^n\slice\Delta^\op}
          (U\circ\partial\tau) \ar[r, "\texttt{local-iso}"] \ar[d] &
        \push_{\smallint K\slice\Delta^\op} U \ar[d, equals] \\
      \Delta^n \otimes V \ar[r, "\texttt{fixed-idx}"] &
        \push_{\smallint \Delta^n\slice\Delta^\op}(U\circ\tau) 
          \ar[r, "\texttt{local-iso}"]  &
        \push_{\smallint K\slice\Delta^\op} U.
    \end{tikzcd}
  \end{equation*}
  Now, specialising to the case $\partial\tau = \sigma$, we obtain an identification
  \begin{equation*}
    \left\{
      \begin{tikzcd}
        \partial\Delta^n \otimes V  \ar[d] \ar[r, "\sigma"] &
          \push_{\smallint K\slice\Delta^\op} U  \\
        \Delta^n \otimes V \ar[ur]
      \end{tikzcd}
    \right\}
    \hspace{16em}
  \end{equation*}
  \begin{equation*}
    \hspace{6em}
    \quad\cong\quad
    \left\{
      \begin{tikzcd}
        \push_{\smallint \Delta^n\slice\Delta^\op}(U\circ\sigma) 
          \ar[r] \ar[d] &
          \push_{\smallint K\slice\Delta^\op} U \\
        \push_{\smallint \Delta^n\slice\Delta^\op} (U\circ\tau) 
          \ar[ur, "\texttt{local-iso}" description]
      \end{tikzcd}
      \mid\; 
      V\subseteq U_\tau
    \right\}
  \end{equation*}
  %
  % Anyone know a better way to do this?
  %
  The local lifting property for \eqref{lifting-property/dhi} states that the set of $V$ for which a diagram of the left-hand form exists covers $U_\sigma$; therefore this set indexes a set of diagrams of the right-hand form whose $U_\tau$ cover $U_\sigma$, which is the local l.p.~for \eqref{lifting-property/simplex}.
  Conversely, the local l.p.~for \eqref{lifting-property/simplex} implies liftings for \eqref{lifting-property/dhi} for $V=U_\tau$ covering $U_\sigma$.
  \qedhere

\end{proof}

\begin{lemma}[Powers of Cartesian fibrations]
\label{util/fibration/power}

  Let $f:C\rightarrow D$ be a Cartesian fibration in $\Cat_1$. Then the exponential $f\circ-:C^{\Delta^1}\rightarrow D^{\Delta^1}$ is a Cartesian fibration, with Cartesian morphisms $u:A\rightarrow B$ those of the form
  \begin{equation*}
    \begin{tikzcd}
      A_0 \ar[r, "u_0"] \ar[d] & B_0 \ar[d] \\
      A_1 \ar[r, "u_1"] & B_1
    \end{tikzcd}
  \end{equation*}
  with $u_0$, $u_1$ both $f$-Cartesian.
  
\end{lemma}
\begin{proof}

  Assume $u:A\rightarrow B$ is Cartesian on each component. Given a map $u':A'\rightarrow B$ and a factorisation $f(A')\rightarrow f(A)\stackrel{f(u)}{\rightarrow} f(B)$ of $f(u')$,
  the Cartesian property for $u_0$ and $u_1$ gives us a unique factorisation
  \begin{equation*}
    \begin{tikzcd}
      A_0' \ar[r] \ar[d] & A_0 \ar[r, "u_0"] \ar[d] \ar[dr, phantom, "\circlearrowright" description] & B_0 \ar[d] \\
      A_1' \ar[r] \ar[ur, phantom, "?" description] & A_1 \ar[r, "u_1"] & B_1
    \end{tikzcd}
  \end{equation*}    
  whereupon it will suffice to prove that the left-hand square is commutative. 

  This follows from the fact that either path around the square is a factorisation of $A'_0\rightarrow B_1$ by $u_1$ lying over $f(A'_0) \rightarrow f(A_1)\rightarrow f(B_1)$, and hence unique as such by the Cartesian property of $u_1$.
  \qedhere
  
\end{proof}

\section{Nerves of atlases}
\label{nerve}

To generate a hypercover --- a kind of simplicial diagram --- from an atlas --- which may be indexed by any category --- we need a kind of `nerve' construction that takes in an arbitrary diagram of open sets and outputs a diagram indexed by (the category of simplices of) a simplicial set, and which transforms the atlas condition into the hypercover condition. 

\begin{remark}[Why not just use the standard nerve?]

  It is easy enough to see that the standard nerve construction \cite[9]{HTT} won't work. Let $I$ be a category, and let $f:\{0,1\}\rightarrow I$ be a functor. This functor admits an extension to $[0\rightarrow 1]$ if and only if $f(0)\rightarrow f(1)$. The most basic examples of atlases violate this (e.g.~\ref{atlas/example/basic}).

\end{remark}

We need to associate a more flexible family of categories to the standard simplices $\Delta^k$. As we will see, the tautological test functor $\nerve$ provided by Grothendieck's theory of test categories \cite{Maltsiniotis} is good enough.

\begin{para}[Tautological nerve]
\label{nerve/of-category/definition}

  If regarded as a functor from $s\Set$ into the full subcategory $\Cat_1\times_{\Groupoid}\Set$ of $\Cat_1$ spanned by the categories whose objects have no non-identity automorphisms, Grothendieck integration admits a right adjoint 
  \begin{equation}
  \label{nerve/formula}
    \nerve: \Cat_1\times_\Groupoid \Set \rightarrow s\Set, \quad
      K \mapsto \Hom(\smallint(-),K),
  \end{equation}
  which is (opposite to) the tautological \emph{test functor} attached to the Grothendieck test category $\Delta$ \cite[Def.~1.3.7]{Maltsiniotis} and \cite[Prop.~1.5.13]{Maltsiniotis}. 
  
\end{para}

\begin{para}[Tautological refinement]
\label{nerve/refinement}
  
  Applying the nerve and then integrating again we obtain a \emph{tautological refinement}:
  \begin{equation}
  \label{nerve/counit}
    \epsilon_J: \smallint\nerve (J) \rightarrow J
  \end{equation}
  where the map $\epsilon_J$ evaluates a simplex $\smallint\Delta^n\rightarrow J$ at its initial object.
  The composite functor $\smallint\circ\nerve$ is a right adjoint to the forgetful functor from the category of left fibrations over $\Delta^\op$
  \begin{equation}
  \label{test/adjunction/variant}
    \smallint\nerve \;:\; \Cat_1\times_{\Groupoid}\Set \;\rightleftarrows\; \mathrm{LFib}(\Delta^\op) \;:\; \forget 
  \end{equation}

\end{para}

\begin{lemma}[Slices of the refinement]
\label{nerve/refinement/slice}
  
  The square
  \begin{equation*}
    \begin{tikzcd}
      \int\nerve(i\downarrow I) \ar[r]  \ar[d, "\epsilon"] &
        \int\nerve(I) \ar[d, "\epsilon"] \\
      i\downarrow I  \ar[r]   & I
    \end{tikzcd}
  \end{equation*}
  is a pullback in $\Cat_1$. Hence, the slice category $i\downarrow_I\int\nerve(I)$ of the refinement may be identified with the refinement $\int\nerve(i\downarrow I)$ of the slice.

\end{lemma}
\begin{proof}

  We may check this is a pullback by factorising
  \begin{equation*}
    \begin{tikzcd}
      \int\nerve(i\downarrow I) \ar[r]  \ar[d] &
        \int\nerve(I) \ar[d] \\
      i\downarrow I\times\Delta^\op  \ar[r] \ar[d] 
        \ar[dr, phantom, "\lrcorner" very near start]  & 
        I \times\Delta^\op \ar[d] \\
      i\downarrow I  \ar[r]   & I 
    \end{tikzcd}
  \end{equation*}
  and observing that the upper square is a commuting square of co-Cartesian fibrations over $\Delta^\op$, whence the pullback property can be checked fibre-wise. Hence, restricting to $\Delta^n:\Delta^\op$, we get a square
  \begin{equation*}
    \begin{tikzcd}
      \Fun(\smallint\Delta^n, i\downarrow I) \ar[r] \ar[d] &
        \Fun(\smallint\Delta^n, I) \ar[d] \\
      i\downarrow I \ar[r] &
        I
    \end{tikzcd}
  \end{equation*}
  to which we apply Lemma \ref{util/cone/initial-object/pullback}.
  \qedhere
  
\end{proof}

\begin{prop}
\label{nerve/counit/is-cofinal}

  The tautological refinement $\epsilon_\Delta$ is $\infty$-cofinal.

\end{prop}
\begin{proof}

  Let $i:I$. We must show that $i\downarrow_I\smallint\nerve(I)$ is weakly contractible \cite[Thm.~4.1.3.1]{HTT}.
  Using Lemma \ref{nerve/refinement/slice} we identify this with $\smallint\nerve(i\downarrow I)$.
  The latter is weakly contractible because $i\downarrow I$ is weakly contractible and $\nerve \dashv \smallint$ induce inverse equivalences of homotopy categories (as $\Delta$ is a test category \cite[Prop.~1.6.14]{Maltsiniotis}).
  \qedhere
  
\end{proof}

\begin{para}[Refinement of a diagram]
\label{nerve/of-diagram/definition}
  
  Let $U:I\rightarrow\opens(X)$ be a diagram. Precomposing $U$ with the counit $\epsilon$, we obtain an indexed diagram
  \begin{equation}
    \begin{tikzcd}
      \int\nerve(I) \ar[r, "U\circ\epsilon"] \ar[d] &
        \opens(X) \\
      \Delta^\op 
    \end{tikzcd}
  \end{equation}
  Denote the associated simplicial semi-representable presheaf by 
  \begin{equation}  
    \nerve( U)_\bullet \defeq 
      \push_{\smallint\nerve(I)\slice\Delta^\op}(U\epsilon)
  \end{equation}
\end{para}

\begin{lemma}
\label{nerve/fill-condition}

  Let $K\rightarrow L$ be a morphism in $s\Set$.
  A diagram $U:I\rightarrow\opens(X)$ of open sets has the local r.l.p.~against $\smallint K\slice \smallint L\slice\Delta^\op$ if and only if $(\smallint\nerve(I), U\epsilon)$ has the local r.l.p.~against $\smallint K\slice \smallint L$.
  
\end{lemma}
\begin{proof}

  The adjunction \eqref{test/adjunction/variant} applied to the arrow category (cf.~Lemma \ref{util/fibration/power}) identifies the sets
  \[
    \left\{
      \begin{tikzcd}
        \smallint K \ar[r] \ar[d] & \smallint\nerve(I) \ar[d] \\
        \smallint L \ar[r] \ar[ur, dashed] & \Delta^\op 
      \end{tikzcd}
    \right\} 
    \quad \cong \quad
    \left\{
      \begin{tikzcd}
        \smallint K \ar[r] \ar[d] & I \ar[d] \\
        \smallint L \ar[r] \ar[ur, dashed] & \point 
      \end{tikzcd}
    \right\}
  \]
  as well as the vertex 
  \[
    (U\circ \epsilon)_\tau = \lim(U\circ\epsilon\circ\tau) = 
      U_{\epsilon\circ\tau}
  \]
  for any lift $\tau$ of either side. 
  \qedhere

\end{proof}

\begin{theorem}
\label{nerve/theorem}

  Let $U:I\rightarrow\opens(X)$ be a diagram. The following are equivalent: \begin{enumerate}
    
    \item $U$ is an atlas.
    
    \item $U$ has the local r.l.p.~for $\int\partial\Delta^n\slice \int\Delta^n$, $n:\N$.
    
    \item $\nerve(U)_\bullet$ is a hypercover.
    
  \end{enumerate}
  In particular, every atlas restricts along an $\infty$-cofinal functor to the index diagram of some hypercover.

\end{theorem}
\begin{proof}

  The equivalence 1 $\Leftrightarrow$ 2 is Proposition \ref{lifting-property/atlas}. Now, 2 $\Leftrightarrow$ 3 by the chain of logical equivalences:
  \begin{itemize}
    \item[] $U$ satisfies the local filling conditions for $\int\partial\Delta^n\slice \int\Delta^n$.
    
    \item[$\Leftrightarrow$] $\smallint\nerve(U)$ satisfies the local filling conditions for $\smallint\partial\Delta^n\slice \smallint\Delta^n\slice\Delta^\op$ (Lemma \ref{nerve/fill-condition}).
    
    \item[$\Leftrightarrow$] $\nerve(U)$ is a hypercover (Proposition \ref{hypercover/fill-condition}). 
    \qedhere
  \end{itemize}
  
\end{proof}

\printbibliography
\end{document}